\documentclass[12pt]{amsart}

\usepackage{comment}
\usepackage[normalem]{ulem}
\usepackage[colorlinks,linkcolor=red,anchorcolor=blue,citecolor=blue]{hyperref} 
\usepackage{amscd} 
\usepackage{amsmath}
\usepackage{mathtools}
\usepackage{amssymb} 
\usepackage{amsthm}
\usepackage{dsfont}
\usepackage{bigdelim}
\usepackage{color} 
\usepackage{enumerate}
\usepackage{graphicx}
\usepackage{mathrsfs}
\usepackage{multirow}
\usepackage[all]{xy} 
\usepackage{color} 
\usepackage{enumitem}
\usepackage[dvipsnames]{xcolor}
\usepackage{bbm}
\usepackage{tikz}
\usepackage{tikz-cd}

\newtheorem{thm}{Theorem}[section] 
\newtheorem{theo}[thm]{Theorem}

\newtheorem{prop}[thm]{Proposition}

\newtheorem{lem}[thm]{Lemma}
\theoremstyle{definition} 
\newtheorem{defn}[thm]{Definition}

\theoremstyle{remark}
\newtheorem{rem}[thm]{Remark}

\newtheorem{claim}[thm]{Claim}

\numberwithin{equation}{section}

\newcommand{\rk}[0]{\operatorname{rk}}
\newcommand{\supp}[0]{\operatorname{Supp}}

\newcommand{\codim}[0]{\operatorname{codim}}

\newcommand{\Sym}{{\rm Sym}}
\newcommand{\End}[0]{\operatorname{End}}

\newcommand{\Tor}{{\rm Tor}}

\newcommand{\reg}{{\rm{reg}}}

\newcommand{\underalign}[2]{\quad \underset{\mathclap{\strut #1}}{#2}\quad}

\newcommand{\tor}{{\rm Tor}}
\newcommand{\Sing}{{\rm Sing}}

\newcommand{\C}{\mathbb{C}}
\newcommand{\Q}{\mathbb{Q}}

\newcommand{\Amp}{{\rm{Amp}}}

\title[HYM metrics on klt KE varieties and MY equality for big $-K_X$, K-stable varieties]{Admissible HYM metrics on klt KE varieties and the MY equality for big anticanonical K-stable varieties}

\author{Satoshi Jinnouchi}
\address{Department of Mathematics, Graduate School of Science, The University of Osaka,
1-1, Machikaneyama-cho, Toyonaka, Osaka 560-0043, Japan.}
\email{{\tt u122988d@ecs.osaka-u.ac.jp}}
\email{{\tt 20160312sti@gmail.com}}

\date{\today}


\begin{document}

\begin{abstract}
This short note includes three results: $(1)$ If a reflexive sheaf $\mathcal{E}$ on a log terminal K\"{a}hler-Einstein variety $(X,\omega)$ is slope stable with respect to a singular K\"{a}hler-Einstein metric $\omega$, then $\mathcal{E}$ admits an $\omega$-admissible Hermitian-Yang-Mills metric. $(2)$ If a K-stable log terminal projective variety with big anti-canonical divisor satisfies the equality of the Miyaoka-Yau inequality in the sense of \cite{IJZ25}, then its anti-canonical model admits a quasi-\'{e}tale cover from $\mathbb{C}P^n$. $(3)$ There exists a holomorphic rank 3 vector bundle on a compact complex surface which is semistable for some nef and big line bundle, but it is not semistable for any ample line bundles.
\end{abstract}
\maketitle
\tableofcontents
\section{Introduction}
This short note includes three results.
The first one is a generalization of \cite{Jin25-2}.
\begin{thm}[=Theorem \ref{lt HE}]\label{lt HE intro}
    Let $X$ be a compact normal analytic variety with klt singularities. Suppose that $X$ admits a singular K\"{a}hler-Einstein metric $\omega$. Then, if a reflexive sheaf $\mathcal{E}$ on $X$ is $\{\omega\}^{n-1}$-slope stable, then $\mathcal{E}$ admits an $\omega$-admissible Hermitian-Yang-Mills metric.
\end{thm}
Second one is the structure theorem of a projective variety with big anti-canonical divisor:
\begin{thm}[=Theorem \ref{BG ineq}]\label{BG ineq intro}
Let $X$ be an $n$-dimensional projective klt variety smooth in codimension $2$. Assume that $-K_X$ is big and that $X$ is K-stable. If the equality
\begin{equation}
\label{eq-MY-equality intro}
\bigl(2(n+1)c_2(X) - nc_1(X)^2 \bigr) \cdot \langle c_1(-K_{X})^{n-2} \rangle = 0
\end{equation}
holds, then the anticanonical model $Z$ admits a quasi-étale cover from $\mathbb{C}\mathbb{P}^n$. $($In this case, by \cite{Xu23}, the anticanonical model $Z$ exists.$)$
\end{thm}
Third one gives an example of a vector bundle that can be treated as a semistable object only when semistability is considered with respect to nef and big classes.
\begin{thm}[=Proposition \ref{eg nef}]\label{eg nef stable}
   Let $E$ be a holomorphic vector bundle on $\mathbb{C}P^2$ defined by the following short exact sequence
   $$
   0\to\mathcal{O}_{\mathbb{C}P^2}\to E\to\mathcal{J}_{\{x_1,x_2\}}\otimes\mathcal{O}_{\mathbb{C}P^2}(2)\to0.
   $$
   Let $\pi:X\to\mathbb{C}P^2$ be a blow up at a point and $C$ be the exceptional curve. Then a holomorphic vector bundle $F:=\pi^*E\oplus(\pi^*\mathcal{O}(1)\otimes\mathcal{O}(C))$ on $X$ satisfies the followings:\\
   $(1)$ $F$ is $c_1(\pi^*\mathcal{O}(1))$-semistable.\\
   $(2)$ For any ample line bundle $H$ on $X$, $F$ is not $c_1(H)$-semistable.\\
   $(3)$ $F$ satisfies the Bogomolov-Gieseker inequality.
\end{thm}
\begin{rem}
    (About Theorem \ref{lt HE intro}) There are several results on the existence of HYM metrics if K\"{a}hler metrics have at most orbifold singularities (eg, \cite{Faulk22},\cite{CGNPPW23}, \cite{FO25}) or conical singularities (eg, \cite{Li00}, \cite{JL25}). In particular, \cite{FO25} proved the existence of HYM metrics for stable reflexive sheaves on klt varieties with orbifold K\"{a}hler metrics via orbifold partial resolution of \cite{Ou25}.
    We remark that the log terminal K\"{a}hler-Einstein metric $\omega$ is orbifold smooth on the orbifold singularities of $X$ (refer to \cite[Remark 1.1]{GPP25} and references in \cite{GPP25}). On the other hand, the singularities of $\omega$  around the higher codimensional singular locus of $X$ have not yet been identified (refer to \cite[around Theorem C]{BG14} and papers that cite \cite{BG14}). Theorem \ref{lt HE intro} is proved by slightly modifying the argument in \cite{Jin25-2} and the proof is independent of the existence of model metrics of $\omega$ around singularities.\\
    (About Theorem \ref{BG ineq intro}) There are several studies about the Miyaoka-Yau inequality and the structure theorem in the case that the equality holds (refer to the introduction of \cite{IJZ25} and references therein for detailed history and related developments concerning the Miyaoka-Yau inequality). We briefly summarize the references of known results. If $X$ is smooth and $K_X$ is ample, refer to \cite{Miy77} when $n=\dim X=2$ and \cite{Yau77} when $X$ is of arbitrary dimension. If $K_X$ is nef, refer to \cite{GKPT19} when $X$ is klt and \cite{GT22} when $(X,D)$ is dlt. If $X$ is projective klt and $-K_X$ is nef, refer to \cite{GKP22}. If $X$ is klt and $K_X$ is big, or $-K_X$ is big and $X$ is $K$-semistable, refer to \cite{IJZ25}. In particular, the authors only proved the inequality in \cite{IJZ25}. Theorem \ref{BG ineq intro} in this paper proves the structure theorem of projective varieties with $-K_X$ big. The remaining open problem is the characterization of the equality of the Miyaoka-Yau inequality when $K_X$ is big. Its resolution will require the Kobayashi-Hitchin correspondence of reflexive Higgs sheaves on singular varieties. The generalizations to klt pairs and dlt pairs are also remaining problems.\\
    (About Theorem \ref{eg nef stable}) In \cite{IJZ25}, the authors proved that the Bogomolov-Gieseker inequality holds for a reflexive sheaf $\mathcal{E}$ that is slope semistable with respect to a big class $\alpha$.  If $\alpha$ is nef and big, and $\mathcal{E}$ is $\alpha^{n-1}$-slope stable, then the Bogomolov-Gieseker inequality of $\mathcal{E}$ with respect to $\alpha$ follows easily from the Bogomolov-Gieseker inequality with respect to $\alpha+\varepsilon\omega$, where $\omega$ is a K\"{a}hler class. Indeed, it is known that $\mathcal{E}$ is slope stable for $\alpha+\varepsilon\omega$ for sufficiently small $\varepsilon>0$ (eg, \cite[Lemma 4.8]{Jin25-2}). The proof for semistable sheaves is more complicated than the stable case, since one cannot prove the openness of semistability. Theorem \ref{eg nef stable} provides an explicit example of a vector bundle whose semistability fails to be an open condition. Theorem \ref{eg nef stable} exhibits a vector bundle that can be treated as a semistable object only when semistability is considered with respect to nef and big classes.
\end{rem}
\section*{Acknowledgments}
The author would like to thank the supervisor Ryushi Goto for encouraging to write this paper. The author also thanks Prof. Masataka Iwai for his valuable comments about Miyaoka-Yau inequality. The author would like to thank Prof. Yuji Odaka for introducing the important paper \cite{CCHTST25}.
\section{Preliminaries}
In this paper, normal analytic varieties are always reduced and irreducible.
\subsection{Sobolev inequality and mean value inequality}
\begin{defn}
    Let $X$ be a compact normal analytic variety. Let $\{U_i\}_{i=1,\ldots k}$ be an open cover of $X$ with closed embeddings $\varphi_i:U_i\hookrightarrow\Omega_i\subset \mathbb{C}^{N_i}$ to open subsets $\Omega_i$ in $\mathbb{C}^{N_i}$.
    A smooth K\"{a}hler metric on $X$ is a smooth K\"{a}hler metric $\omega$ on $X_{\reg}$ such that, for any $i$, there is a smooth K\"{a}hler metric $\omega_{\Omega_i}$ on $\Omega_i$ such that $\omega|_{U_i\cap X_{\reg}}=\varphi_i^*\omega_{\Omega_i}|_{X\reg}$.We can define a closed positive $(1,1)$-current on $X$ by the same way. \\
    For a closed positive $(1,1)$-current $T$ on $X$, there is a plurisubharmonic function $\varphi_i$ on $\Omega_i$ such that $T|_{U_i}=dd^c\varphi_i|_{U_i}$. We say that $\varphi_i$ is a local potential of $T$. If any local potential $\varphi_i$ is locally bounded for any $i$, we say that $T$ is a closed positive $(1,1)$-current with bounded potentials.
\end{defn}
The readers can refer to \cite{CCHTST25} for the definition of singular K\"{a}hler-Einstein metrics on klt varieties. 
\begin{defn}[\cite{CCHTST25} Definition 1.4]\label{tame app}
    Let $\omega$ be a closed positive $(1,1)$-current with bounded potentials on a compact normal variety $X$ such that $\omega$ is smooth on $X_{\reg}$ which is cohomologas to a smooth K\"{a}hler metric $\omega_X$ on $X$. Let $\pi:Y\to X$ be a resolution of singularities that is biholomorphic over $X_{\reg}$ and $\omega_Y$ be a smooth K\"{a}hler metric on $Y$. We say that a sequence of smooth K\"{a}hler metrics $\omega_i$ in $\{\pi^*\omega_X\}+(1/i)\{\omega_Y\}$ is a tame approximation of $\omega$ if we have the following:
    \begin{itemize}
        \item If we denote by $\omega_i=\pi^*\omega_X+\omega_Y+dd^cu_i$, then $\omega_i\to\pi^*\omega$ locally smoothly on $\pi^{-1}(X_{\reg})$ and 
        \item there are constants $C>0$ and $p>1$ independent of $i$ such that 
        $$
        \sup_Y|u_i|<C \hbox{,  } \int_Y\left(\frac{\omega_i^n}{\omega_Y^n}\right)^p\omega_Y^n<C.
        $$
    \end{itemize}
\end{defn}

The most important property of the klt K\"{a}hler-Einstein metrics in this paper is the following:
\begin{thm}[\cite{CCHTST25} Theorem 1.2, Proposition 3.1, \cite{GPSS23} Theorem 2.1]\label{sob ineq}
    Suppose that $(X,\omega_X)$ is a compact normal K\"{a}hler variety with a smooth K\"{a}hler metric $\omega_X$. Suppose that $\omega=\omega_X+dd^cu$ is a closed positive $(1,1)$-current with bounded potential satisfying the following properties:
    \begin{itemize}
        \item $u$ is smooth on $X_{\reg}$,
        \item $\omega^n=e^F\omega_X$ where $e^F\in L^p(\omega_X^n)$ for some $p>1$ and $F\in L^1(\omega_X^n)$.
        \item ${\rm{Ric}}(\omega)\ge -C(\omega+\omega_X)$ on $X_{\reg}$ in the sense of current. 
    \end{itemize}
    Then $\omega$ satisfies the followings:
    \begin{enumerate}
        \item[$(1)$] \cite[Theorem 1.2]{CCHTST25} $\omega$ is a K\"{a}hler current on $X$, that is, there is a constant $C>0$ such that $\omega\ge C^{-1}\omega_X$ holds on $X$ as a current.
        \item[$(2)$] \cite[Proposition 3.1]{CCHTST25} We can find a tame approximation of $\omega$ with smooth K\"{a}hler metrics $\omega_i$ in $\pi^*\{\omega_X\}+\frac{1}{i}\{\omega_Y\}$ on a resolution $\pi:Y\to X$.
        \item[$(3)$] \cite[Theorem 2.1]{GPSS23} The tame approximation $\omega_i$ in $(2)$ above satisfies the uniform Sobolev inequality, that is, there are constants $q>1$ and $C>0$ independent of $i$ such that 
        $$
        \left(\frac{1}{\{\omega_i\}^n}\int_Y|u-\overline{u}|^{2q}\omega_i^n\right)^{1/q}
        \le \frac{C}{\{\omega_i\}^n}\int_Y|\nabla u|_{\omega_i}^2\omega_i^n
        $$
        holds for any $u\in L^2_1(X,\omega_i)$ where $\overline{u}=\frac{1}{\{\omega_i\}^n}\int_Yu\omega_i^n$.
        \item[$(4)$] {\rm{(\cite[Lemma 6.3]{GPSS23}, see also \cite[proof of Proposition 4.13]{Jin25-2})}} If a sequence of nonnegative smooth functions $u_i\in C^{\infty}(Y)$ satisfies $\Delta_{\omega_i}u_i\ge -N$ for some constant $N>0$, then there exists a constant $C>0$ such that the following holds:
        $$
        \sup_Yu_i\le C\int_Yu_i\omega_i^n.
        $$
    \end{enumerate}
\end{thm}
\subsection{Positive product}
In this section, we briefly introduce the positive products (or movable intersection products) of big classes.
\begin{defn}
    Let $X$ be a compact normal variety and $\alpha\in H^{1,1}_{BC}(X)$ be a Bott-Chern cohomology class (see eg, \cite[\S 2.2.3]{IJZ25}). Then\\
    $(1)$ $\alpha$ is a K\"{a}hler class if it is represented by a smooth K\"{a}hler metric.\\
    $(2)$ $\alpha$ is a big class if $\alpha$ is represented by a K\"{a}hler current. Here a closed positive $(1,1)$-current $T$ is a K\"{a}hler current if $T\ge -\varepsilon\omega$ where $\omega$ is a positive definite smooth hermitian form on $X$.\\
    $(3)$ Assume that $X$ is smooth. For a big class $\alpha$, the ample locus of $\alpha$ is defined by
    $$
    \Amp(\alpha):=\{x\in X\mid \hbox{There is a K\"{a}hler current $T\in\alpha$ that is smooth around $x$.} \}
    $$ 
    The complement $E_{nK}(\alpha):=X\setminus \Amp(\alpha)$ is the nonK\"{a}hler locus of $\alpha$. By \cite{Bou04}, the ample locus $\Amp(\alpha)$ is a Zariski open set.
\end{defn}
\noindent If $\alpha=\pi^*\omega$ where $\pi:X\to Y$ is a proper bimeromorphic morphism and $\omega\in H^{1,1}_{BC}(Y)$ is a K\"{a}hler class, then $E_{nK}(\alpha)$ coincides with the $\pi$-exceptional divisor.\\
Assume that $X$ is a compact K\"{a}hler manifold and $\alpha$ is a big class on $X$. Following to \cite[Theorem 3.5]{BDPP13}, we introduce the positive product of $\alpha$: We choose a K\"{a}hler current $T$ in $\alpha$ whose local potentials have logarithmic singularities and a log resolution $\mu:\widetilde{X}\to X$ such that
\begin{equation}\label{appro zariski decomp}
\mu^*T=\beta+[D]
\end{equation}
where $\beta$ is a smooth semipositive form and $[D]$ is an integral current along an effective divisor. Then, for any semi-positive class $u\in H^{n-k,n-k}(X,\mathbb{R})$, there is a constant $C_u$ independent of $T$ such that $\int_{\widetilde{X}}\beta^k\wedge u\le C_u$. Thus we can find a sequence of K\"{a}hler currents $T_j$ in $\alpha$ such that
\begin{equation}\label{positive product app}
\lim_{j\to\infty}\int_{\widetilde{X}_j}\beta_j\wedge \mu_j^*u=\sup_{T\in\alpha}\int_{\widetilde{X}}\beta\wedge \mu^*u.
\end{equation}
Here $\mu_j:\widetilde{X}_j\to X$ is a modification such that $\mu_j^*T_j=\beta_j+[D_j]$ and, supremum in RHS is taken over closed positive $(1,1)$-currents $T\in\alpha$ and $\mu:\widetilde{X}\to X$ is a modification such that $\mu^*T=\beta+[D]$ as above.
Therefore, if we choose a basis $u_1\ldots,u_l$ of $H^{n-k,n-k}(X,\mathbb{R})$ consisting of positive classes as $u$, we can find a subsequence of $T_j$ such that (\ref{positive product app}) holds for any positive class $u$. Then, by Poincar\'{e} duality, we can find a limit
$\langle\alpha^{k}\rangle:=\lim_{j\to\infty}\mu_{j,*}(\beta_j^k)\in H^{k,k}(X,\mathbb{R})$, which is the positive product of $\alpha$.
\begin{rem}
    The positive product $\langle\alpha^k\rangle$ is independent of the choices of $T_j$ and the basis $u_1,\ldots,u_l$. In fact, $\langle\alpha^k\rangle$ coincides with $\{\langle T_{\min}^k\rangle\}$ which is the de-Rham cohomology class of the nonpluripolar product of a closed positive $(1,1)$-current $T_{\min}$ in $\alpha$ with minimal singularities \cite[Proposition 1.18]{BEGZ10}. And $\{\langle T_{\min}^k\rangle\}$ is independent of the choices of $T_{\min}$ \cite[Theorem 1.16]{BEGZ10}.
\end{rem}
The decomposition (\ref{appro zariski decomp}) gives an approximate Zariski decomposition. Following lemma says that if a big class admits a Zariski decomposition, its positive product coincides with the positive part of the Zariski decomposition.
\begin{lem}[\cite{DHY23} Lemma A.5, see also \cite{IJZ25} Proposition 3.15]\label{mov prod resol}
    Let $X$ be a compact K\"{a}hler manifold and $Y$ be a compact normal variety. Let $\alpha$ be a big class on $X$ and $\omega\in H^{1,1}_{BC}(Y)$ be a K\"{a}hler class on $Y$. Suppose that there exists a proper bimeromorphic morphism $\pi:X\to Y$ and an effective $\pi$-exceptional divisor $D$ such that $\alpha=\pi^*\omega+[D]$ holds. Then $\langle\alpha^k\rangle=(\pi^*\omega)^k$ holds.
\end{lem}

\subsection{Slope stability and the HYM metrics}
In this section, we briefly introduce the slope stability and the Hermitian-Yang-Mills metrics on singular varieties (cf. \cite{IJZ25}). For any bimeromorphic morphism $\pi:Y\to X$ and coherent sheaf $\mathcal{E}$ on $X$, we denote the reflexive pull back of $\mathcal{E}$ by $\pi^{[*]}\mathcal{E}:=(\pi^*\mathcal{E})^{\vee\vee}$.
\begin{defn}
    Let $X$ be a compact normal analytic variety and $\alpha\in H^{1,1}_{BC}(X)$ be a big class on $X$. We fix a resolution of singularities $\pi:Y\to X$. \\
    $(1)$ The $\langle\alpha^{n-1}\rangle$-slope of a torsion free coherent sheaf $\mathcal{E}$ is defined by
    $$
    \mu_{\langle\alpha^{n-1}\rangle}(\mathcal{E}):=\frac{1}{\rk\mathcal{E}}\int_Yc_1(\pi^{[*]}\mathcal{E})\wedge\langle(\pi^*\alpha)^{n-1}\rangle.
    $$
    $(2)$ A reflexive coherent sheaf $\mathcal{E}$ is $\langle\alpha^{n-1}\rangle$-slope (semi) stable if for any torsion free subsheaf $\mathcal{F}$ of $\mathcal{E}$ with $0<\rk\mathcal{F}<\rk\mathcal{E}$ that has the torsion free quotient $\mathcal{E}/\mathcal{F}$, the inequality $\mu_{\alpha^{n-1}}(\mathcal{F})<(\le)\mu_{\alpha^{n-1}}(\mathcal{E})$ holds.
\end{defn}
A Moishezon variety is a complex analytic variety which is bimeromorphic to a complex projective manifold.
\begin{lem}[\cite{Jin25}, \cite{IJZ25} Lemma 4.14]
    If $X$ is Moishezon or $\alpha$ satisfies the vanishing property in the sense of \cite[Definition 3.23, Lemma 3.24]{IJZ25}, the above definition of slope (semi) stability is independent of the choices of resolutions $\pi$.
\end{lem}
\begin{rem}
   On compact normal Moishezon variety, any big class satisfies the vanishing property by \cite{Nystr19}. The vanishing property is closely related to the differentiability of the volume function $\alpha\mapsto \langle\alpha^n\rangle$ (see eg, \cite{Nystr19} and \cite{Vu23}).
\end{rem}
Next we define the notion of the admissible HYM metrics for closed positive $(1,1)$-currents. For a coherent sheaf $\mathcal{E}$ on a compact normal variety $X$, we denote by $\Sing(\mathcal{E})$ the minimal analytic subset of $X$ such that $\mathcal{E}$ is locally free on the Zariski open set $X\setminus \Sing(\mathcal{E})$. If $\mathcal{E}$ is torsion free, then $\Sing(\mathcal{E})$ is of codimension at least 2.
\begin{defn}[\cite{Jin25-2}]
    Let $X$ be a compact normal analytic variety and $T$ be a closed positive $(1,1)$-current on $X$ which is smooth K\"{a}hler on a Zariski open set $\Omega_T$ in $X$. Let $\mathcal{E}$ be a reflexive sheaf on $X$. Then a $T$-admissible Hermitian-Yang-Mills metric on $\mathcal{E}$ is a smooth hermitian metric on $\mathcal{E}|_{(X_{\reg}\cap\Omega_T)\setminus {\rm{Sing}}(\mathcal{E})}$ such that
    \begin{itemize}
        \item $\sqrt{-1}\Lambda_TF_h=\lambda{\rm{Id}}$ on $(X_{\reg}\cap\Omega_T)\setminus {\rm{Sing}}(\mathcal{E})$ and
        \item $\int_{(X_{\reg}\cap\Omega_T)\setminus {\rm{Sing}}(\mathcal{E})}|F_h|_{h,T}^2T^n<\infty$.
    \end{itemize}
\end{defn}
\section{Admissible HYM metric for klt K\"{a}hler-Einstein metrics}
\begin{thm}\label{lt HE}
     Let $X$ be a compact normal analytic variety with klt singularities. Suppose that $X$ admits a singular K\"{a}hler-Einstein metric $\omega$. Then, if a reflexive coherent sheaf $\mathcal{E}$ on $X$ is $\{\omega\}^{n-1}$-stable, then $\mathcal{E}$ admits an $\omega$-admissible HYM metric.
\end{thm}
\begin{proof}
    Since the klt KE metric $\omega$ satisfies three conditions in Theorem \ref{sob ineq}, $\omega$ is a K\"{a}hler current on $X$ by Theorem \ref{sob ineq} $(1)$. 
    Let $\pi:Y\to X$ be a strong resolution of $X$, i.e., $\pi$ is a projective birational morphism such that $\pi:\pi^{-1}(X_{\reg})\to X_{\reg}$ is isomorphic and $E:=\pi^*\mathcal{E}/\Tor$ is locally free. Let $\omega_Y$ be a smooth K\"{a}hler metric on $Y$. Then, by Theorem \ref{sob ineq} $(2)$, there exists a sequence of smooth K\"{a}hler metrics $\omega_i$ in $\pi^*\{\omega\}+\frac{1}{i}\{\omega_Y\}$ which is a tame approximation of $\omega$ (Definition \ref{tame app}).
    Let $h_E$ be a smooth hermitian metric on $E$. 
    Here we know that $E$ is $\{\omega_i\}^{n-1}$-stable (e.g. \cite[Lemma 4.8]{Jin25-2}).
    Hence, by the Uhlenbeck-Yau theorem \cite{UY86}, we obtain a positive definite hermitian endomorphism $\Psi_i$ of $(E,h_E)$ such that $h_i=h_E\Psi_i^2$ gives an $\omega_i$-HYM metric on a holomorphic vector bundle $E$ and it defines
    an $\omega_i$-HYM connection $\nabla_i$ on the underlying complex hermitian vector bundle $(E,h_E)$ (refer to \cite[\S 4.2]{Jin25-2} for the detailed construction from an $\omega_i$-HYM metric on $E$ to an $\omega_i$-HYM connection on $(E,h_E)$). 
    By the standard computation (cf.\cite[Proposition 4.13, (4.1)]{Jin25-2}), we have
    \begin{equation}\label{laplacian ineq}
    \Delta_{\omega_i}\log|\Psi_i|^2_{h_E}
    \ge -|\Lambda_{\omega_i}F_{h_E}|_{h_E}-|\Lambda_{\omega_i}F_{\nabla_i}|_{h_E}.
    \end{equation}
    Since $\nabla_i$ is $\omega_i$-HYM, the second term in the RHS above is uniformly bounded by a constant. We estimate the first term. Let $D=\Sigma_{i=1}^ka_iD_i$ with $a_i>0$ be a $\pi$-exceptional effective divisor such that $\pi^*\{\omega\}-[D]$ is a K\"{a}hler class. Let $s_D$ be a defining section of $D$. Then we can find a smooth hermitian metric $h_D$ on $\mathcal{O}(D)$ such that $\pi^*\omega+dd^c\log|s_D|_{h_D}^2=\pi^*\omega-\sqrt{-1}F_{h_D}+[D]$ defines a K\"{a}hler current on $Y$. Hence there is a constant $C>0$ which is independent of $i$ such that
    $$
    -C(\omega_i+dd^c\log|s_D|^2)
    \le \sqrt{-1}F_{h_E} 
    \le C(\omega_i+dd^c\log|s_D|^2).
    $$
    in the sense of currents.
    Therefore we obtain
    $$
    -C(n+\Delta_{\omega_i}\log|s_D|^2)
    \le \sqrt{-1}\Lambda_{\omega_i}F_{h_E}
    \le C(n+\Delta_{\omega_i}\log|s_D|^2).    
    $$
    and hence $|\Lambda_{\omega_i}F_{h_E}|\le C(n+\Delta_{\omega_i}\log|s_D|^2)$. By substituting to (\ref{laplacian ineq}), we obtain
    $$
    \Delta_{\omega_i}(\log|\Psi_i|_{h_E}^2+C\log|s_D|^2)
    \ge -C.
    $$
    Then, by Theorem \ref{sob ineq} $(4)$, we obtain that
    $$
    \log|\Psi_i|^2+C\log|s_D|^2\le C
    $$
    by normalizing $\Psi_i$ so that $\int_X\log|s_D^N\Psi_i|^2\omega_i^n=1$ where we choose $N:=C\in\mathbb{Z}$. In particular we have
    \begin{equation}\label{sup estimate}
        |s_D^N\Psi_i|\le C.
    \end{equation}
    The pullback current $\pi^*\omega$ satisfies the complex Monge-Amp\`{e}r{e} equation $(\pi^*\omega)^n=e^F\omega_Y$ and $e^F$ has at most pole along the $\pi$-exceptional divisor: $e^F=e^f|s_{D_1}|^{2p_1}\cdots|s_{D_k}|^{2p_k}$ where $f\in C^{\infty}(Y)$ and $p_i>0$. Furthermore $\pi^*\omega\ge |s_D|^2\omega_Y$.
    In particular, $\Psi_i$ and $\omega_i$ satisfies \cite[Assumption 4.14]{Jin25-2}. Then, the arguments following Proposition 4.15 in \cite{Jin25-2} applies by replacing $\Psi_i$ to $s_D^N\Psi_i$ with sufficiently large $N>0$.
    Therefore, by Theorem \cite[Theorem 4.22]{Jin25-2}, the above (\ref{sup estimate}) concludes the existence of $\pi^*\omega$-admissible HYM metric on $E$. 
    In particular, if we denote by $\nabla$ the Uhlenbeck limit of $\nabla_i$, then $\overline{\partial}^{\nabla}$ gives a holomorphic structure on the complex vector bundle $E|_{Y\setminus D}$ and $\Psi_{\infty}$, the $L^2$-local limit of $\Psi_i$, gives a holomorphic isomorphism from $(E,\overline{\partial}^{E})$ to $(E,\overline{\partial}^{\nabla})$ over $Y\setminus D$.
\end{proof}

\section{Structure theorem when MY equality holds for big anticanonical K-stable varieties}

\begin{defn}\cite[Definition 3.2 and Proposition 3.7]{GKP22}
\label{defn-MY-1}
Let $X$ be a connected complex manifold, and let $\mathcal{F}$ be a locally free coherent sheaf of rank $r$. We say that $\mathcal{F}$ is \emph{projectively flat} if the following equivalent conditions hold:
\begin{enumerate}
\item There exists a representation $\pi_1(X) \to PGL(r+1,\C)$ such that the associated projective bundle satisfies $\mathbb{P}(\mathcal{F}) \cong \mathbb{P}_{\rho}$, where
$$
 \mathbb{P}_{\rho} \cong X_{\mathrm{univ}} \times \mathbb{P}^r / \pi_{1}(X),
$$
and $X_{\mathrm{univ}}$ denotes the universal cover of $X$. 
\item The sheaf $\End(\mathcal{F})$ admits a flat connection. 
\item The sheaf $\Sym^r \mathcal{F} \otimes \det(\mathcal{F})^{\vee}$ admits a flat connection.
\end{enumerate}
\end{defn}
Let $X$ be a compact normal Moishezon variety and $\alpha\in H^{1,1}_{BC}(X)$ be a big class. Let $\mathcal{E}$ be a reflexive sheaf on $X$. Let us fix a proper bimeromorphic morphism $\pi:Y\to X$ such that $Y$ is smooth K\"{a}hler and $E:=\pi^*\mathcal{E}/\Tor$ is locally free. Following to \cite[Proposition 4.16]{IJZ25}, we define
\begin{equation}\label{def of int prod}
c_1(\mathcal{E})^2\cdot\langle\alpha^{n-2}\rangle:=c_1(E)\cdot\langle(\pi^*\alpha)^{n-2}\rangle, \hbox{ } c_2(\mathcal{E})\cdot\langle\alpha^{n-2}\rangle:= c_2(E)\cdot\langle(\pi^*\alpha)^{n-2}\rangle.
\end{equation}
By \cite[Proposition 4.16]{IJZ25}, the above definition is independent of choices of resolutions $\pi$.
\begin{prop}[cf. \cite{Chen25},\cite{IJZ25}]
\label{prop-semistable}
Let $X$ be a compact normal Moishezon variety smooth in codimension $2$. Let $\mathcal{E}$ be a reflexive sheaf of rank $r$ and $\alpha \in H^{1,1}_{BC}(X)$ be a big class. Suppose that $\mathcal{E}$ is $\langle\alpha^{n-1}\rangle$-slope semistable. Then the Bogomolov-Gieseker inequality
\begin{equation}
\label{eq-BG}
\bigl(2rc_2(\mathcal{E}) - (r-1)c_1(\mathcal{E})^2 \bigr) \cdot \langle\alpha^{n-2}\rangle \ge 0 
\end{equation}
holds. Moreover, if $\alpha$ is big and semiample, $\mathcal{E}$ is $\alpha^{n-1}$-slope stable and the equality in \eqref{eq-BG} holds, then $\mathcal{E}$ is locally free and projectively flat on $X_{\reg} \cap \Amp(\alpha)$.
\end{prop}
\noindent Here a cohomology class $\alpha$ is big and semiample if there is a sequence of blowups $\pi:X\to Y$ where $Y$ is compact normal K\"{a}hler variety and a K\"{a}hler class $\omega\in H^{1,1}_{BC}(Y)$ on $Y$ such that $\alpha=\pi^*\omega$.
\begin{proof}
    The inequality (\ref{eq-BG}) is a result of \cite{IJZ25}. The equality case follows from the argument of \cite{Chen25} as follows: By \cite[Proposition 4.16]{IJZ25}, we can assume that $X$ is a smooth K\"{a}hler manifold and $\mathcal{E}$ is locally free. We fix a smooth hermitian metric $h_0$ on $\mathcal{E}$. Let $Y$ be a compact normal analytic variety with a smooth K\"{a}hler metric $\omega$ and a composition of blowups $\pi:X\to Y$ such that $\alpha=\{\pi^*\omega\}$. Let $\omega_X$ be a smooth K\"{a}hler metric on $X$. Then $\mathcal{E}$ is slope stable with respect to $\omega_{\varepsilon}=\pi^*\omega+\varepsilon\omega_X$. Let $\nabla_{\varepsilon}$ be an $\omega_{\varepsilon}$-HYM connection on the complex hermitian vector bundle $(\mathcal{E}, h_0)$. Since $\int_X|F_{\nabla_{\varepsilon}}|_{h_0,\omega_{\varepsilon}}^2\omega_{\varepsilon}^n\le C$, there exists a $\pi^*\omega$-admissible HYM connection $\nabla$ of $(\mathcal{E}, h_0)|_{\Amp(\pi^*\omega)}$ such that $\nabla_{\varepsilon}$ locally smoothly converges to $\nabla$ up to $h_0$-unitary transformations as in the proof of Theorem \ref{lt HE} (see also \cite[Corollary 48]{CW22}). Since a sequence of Radon measures $d\mu_{\varepsilon}=(2rc_2(\nabla_{\varepsilon})-(r-1)c_1(\nabla_{\varepsilon})^2)\wedge \omega_{\varepsilon}^{n-2}$ converges to a Radon measure $d\mu=(2rc_2(\nabla)-(r-1)c_1(\nabla)^2)\wedge \pi^*\omega^{n-2}$, Fatou's lemma shows that
    \begin{align*}
        &(2rc_2(\mathcal{E})-(r-1)c_1(\mathcal{E})^2)\cdot\alpha^{n-2}\\
        &=\lim_{\varepsilon\to0}\int_X(2rc_2(\nabla_{\varepsilon})-(r-1)c_1(\nabla_{\varepsilon})^2)\wedge \omega_{\varepsilon}^{n-2}\\
        &\ge\int_{\Amp(\pi^*\omega)}(2rc_2(\nabla)-(r-1)c_1(\nabla)^2)\wedge \pi^*\omega^{n-2}\ge 0.
    \end{align*}
    Since the LHS is 0, the connection $\nabla$ is projectively flat. Thus $\mathcal{E}$ is projectively flat on $\Amp(\alpha)$.
\end{proof}
 In \cite{IJZ25}, the authors proved the following:
\begin{theo}[\cite{IJZ25} Theorem 1.1]
Let $X$ be an $n$-dimensional projective klt variety with big anti-canonical divisor $-K_X$ that is $K$-semistable. Then  the following Miyaoka-Yau inequality holds:
$$
\big(2(n+1)\widehat{c_2}(X)-n\widehat{c_1}(X)^2\big)\cdot\langle c_1(-K_X)^{n-2}\rangle\ge 0.
$$
Here $\widehat{c_i}(X)$ is the orbifold Chern classes of $X$. If $X$ is smooth in codimension 2, they coincide with usual Chern classes $c_i(X)$.
\end{theo}
In the following Theorem \ref{BG ineq}, we show that the structure theorem holds when the equality of the Miyaoka-Yau inequality holds.

\begin{thm}[{cf. \cite{His24}}]\label{BG ineq}
Let $X$ be an $n$-dimensional projective klt variety smooth in codimension $2$. Assume that $X$ is K-stable and that $-K_X$ is big. If the equality
\begin{equation}
\label{eq-MY-equality}
\bigl(2(n+1)c_2(X) - nc_1(X)^2 \bigr) \cdot \langle c_1(-K_{X})^{n-2} \rangle = 0
\end{equation}
holds, then the anticanonical model $Z$ admits a quasi-étale cover from $\mathbb{C}\mathbb{P}^n$. $($In this case, by \cite{Xu23}, the anticanonical model $Z$ exists.$)$
\end{thm}
We recall the canonical extension sheaf. 
Let $X$ be a normal analytic variety, and assume that $K_X$ is $\Q$-Cartier. Then, by \cite[Section~4]{GKP22} (see also \cite[Subsection~5.2]{IJZ25}), there exists a reflexive sheaf $\mathcal{E}_{X}$, called the \emph{canonical extension sheaf}, such that the following short exact sequence 
\begin{equation}
\label{eq-canoincal-extension}
0 \to \mathcal{O}_{X} \to \mathcal{E}_{X} \to \mathcal{T}_{X} \to 0
\end{equation}
is locally split, where $\mathcal{T}_{X}$ is the holomorphic tangent sheaf, i.e., the dual of the reflexive cotangent sheaf $\Omega_{X}^{[1]}$.

\begin{proof}
This proof is motivated by \cite{GKP22}. We use the same notation as in \cite[Subsection~5.2]{IJZ25}. Since $X$ is K-stable, it follows from \cite[Theorems 1.1, 1.2, and Corollary 3.5]{Xu23} that the anticanonical model $Z$ is a K-stable klt Fano variety. Moreover, by \cite[Remark 4 in \S 2 and Theorem 9 in \S 3]{DGP24}, the canonical extension sheaf $\mathcal{E}_Z$ is $c_1(-K_Z)^{n-1}$-polystable.

Let $f \colon X \dashrightarrow Z$ be the natural birational map, and take a resolution $W$ with morphisms $p \colon W \to X$ and $q \colon W \to Z$ resolving $f$:
$$
\xymatrix@C=40pt@R=30pt{
& W \ar[ld]_p \ar[rd]^q & \\
X \ar@{-->}[rr]^f && Z
}
$$
By the argument of \cite[Subsection~5.2]{IJZ25}, there exists an effective $q$-exceptional divisor $B$ on $W$ such that
\begin{equation}
\label{eq-contraction}
p^*(-K_X) = q^*(-K_Z) + B
\end{equation}
and $p(\supp(B))$ coincides with the $f$-exceptional locus.
Therefore, we have an isomorphism $p^*\mathcal{E}_X \cong q^*\mathcal{E}_Z$ outside the support of $B$.

\begin{claim}
\label{claim-locallyfree}
The canonical extension sheaf $\mathcal{E}_Z$ is projectively flat on $Z_{\reg}$.
\end{claim}

\begin{proof}[Proof of Claim \ref{claim-locallyfree}]
The sheaf $\mathcal{E}_Z$ is locally free on $Z_{\reg}$ by considering (\ref{eq-canoincal-extension}). Thus it suffices to show projective flatness. By taking a modification, we may assume that $p^{*}\mathcal{E}_X/\tor$ is locally free. By \cite{Jin25} (see also \cite[Theorem~1.2]{IJZ25}), the canonical extension sheaf $\mathcal{E}_X$ is $\langle c_1(-K_X)^{n-1} \rangle$-polystable. Hence $p^{*}\mathcal{E}_X/\tor$ is $\langle p^{*}c_1(-K_X)^{n-1} \rangle$-polystable. By the definition of the intersection product (\ref{def of int prod}) (cf, \cite[Proposition~4.16]{IJZ25}), we obtain
\begin{align*}
& \bigl( (2n+1)c_2(p^{*}\mathcal{E}_X/\tor) - nc_1(p^{*}\mathcal{E}_X/\tor)^2 \bigr)\cdot \langle p^{*}c_1(-K_{X})^{n-2} \rangle \\
&\underalign{\text{\cite{IJZ25}}}{=} \bigl(2(n+1)c_2(\mathcal{E}_{X}) - nc_1(\mathcal{E}_{X})^2 \bigr)\cdot \langle c_1(-K_{X})^{n-2} \rangle \underalign{\eqref{eq-MY-equality}}{=} 0.
\end{align*}
From \eqref{eq-contraction} and Lemma \ref{mov prod resol}, for each $k \in \mathbb{N}$ we have $\langle p^{*}c_1(-K_X)^{k} \rangle = q^{*}c_1(-K_Z)^{k}$. Thus $p^{*}\mathcal{E}_X/\tor$ is $q^{*}c_1(-K_Z)^{n-1}$-polystable and
$$
\bigl(2(n+1)c_2(p^{*}\mathcal{E}_X/\tor) - nc_1(p^{*}\mathcal{E}_X/\tor)^2 \bigr)\cdot q^{*}c_1(-K_Z)^{n-2} = 0.
$$
Hence, by Proposition~\ref{prop-semistable}, $p^{*}\mathcal{E}_X/\tor$ is projectively flat on $\Amp(-q^{*}K_Z)$. Since $\Amp(-q^{*}K_Z) = W \setminus \supp(B)$ and $p^{*}\mathcal{E}_X/\Tor \cong q^*\mathcal{E}_Z$ outside $\supp(B)$, it follows that $q^*\mathcal{E}_Z$ is projectively flat on $W \setminus \supp(B)$.

Since $q^{-1}(Z_{\reg}) \setminus \supp(B) \cong Z_{\reg} \setminus q(\supp(B))$, we conclude that $\mathcal{E}_Z$ is projectively flat on $Z_{\reg} \setminus q(B)$. We have $\codim_{Z_{\reg}} q(B) \ge 2$ since $B$ is $q$-exceptional, which implies
$$
\pi_{1}(Z_{\reg} \setminus q(B)) \cong \pi_{1}(Z_{\reg}).
$$
Thus $\mathcal{E}_Z$ is projectively flat on $Z_{\reg}$.
\end{proof}

Since $Z$ has at most klt singularities, by \cite[Theorem 1.14]{GKP16} (see also \cite[Subsection 2.5]{GKP22}), there exists a finite quasi-étale Galois cover $\nu \colon \widehat{Z} \to Z$ such that $\widehat{Z}$ is maximal quasi-étale. In particular, $-K_{\widehat{Z}} = \nu^{*}(-K_{Z})$ (refer to \cite[\S 5, (5.11)]{GKP22}), so $-K_{\widehat{Z}}$ is ample. 

We will show that $\widehat{Z} \cong \C\mathbb{P}^{n}$.
Let $\widehat{Z}^{\circ} := \nu^{-1}(Z_{\reg})$. Then $\mathcal{E}_{\widehat{Z}}|_{\widehat{Z}^{\circ}}$ is also projectively flat since $\nu \colon \widehat{Z}^{\circ} \to Z_{\reg}$ is étale by purity of the branch locus (refer to \cite[Remark 1.4]{GKP16}), and we have $\mathcal{E}_{\widehat{Z}}|_{\widehat{Z}^{\circ}} \cong \pi^{*}(\mathcal{E}_{Z}|_{Z_{\reg}})$ by \cite[Remark~4.3]{GKP22}. Hence by Definition~\ref{defn-MY-1}, the reflexive sheaf $\End(\mathcal{E}_{\widehat{Z}})$ is flat and locally free on $\widehat{Z}^{\circ}$. As $\codim(\widehat{Z} \setminus \widehat{Z}^{\circ}) \ge 2$, we have $\pi_1(\widehat{Z}_{\reg}) \cong \pi_1(\widehat{Z}^{\circ})$, so $\End(\mathcal{E}_{\widehat{Z}})$ is also flat locally free on $\widehat{Z}_{\reg}$. Since $\widehat{Z}$ is maximally quasi-étale, \cite[Theorem~1.14]{GKP16} implies that $\End(\mathcal{E}_{\widehat{Z}})$ is flat locally free on $\widehat{Z}$.

Consider the following locally split exact sequence:
\begin{equation}\label{can ext of Z}
0 \to \mathcal{O}_{\widehat{Z}} \to \mathcal{E}_{\widehat{Z}} \to \mathcal{T}_{\widehat{Z}} \to 0.
\end{equation}
Since (\ref{can ext of Z}) is locally splittable by the construction \cite[Construction 4.1]{GKP22}, we have that $\mathcal{T}_{\widehat{Z}}$ is a direct summand of $\End(\mathcal{E}_{\widehat{Z}})$. Since $\End(\mathcal{E}_{\widehat{Z}})$ is locally free, so is $\mathcal{T}_{\widehat{Z}}$. Thus, by \cite[Theorem 6.1]{GKKP11}, $\widehat{Z}$ is smooth.
Since $\mathcal{E}_{\widehat{Z}}$ is projectively flat and $\widehat{Z}$ is Fano, there exists a line bundle $L$ with 
$\mathcal{E}_{\widehat{Z}} \cong L^{\oplus (n+1)}$.
(Note that Fano manifolds are simply connected and thus the representation $\pi_1(\widehat{Z})\to PGL(n+1,\mathbb{C})$ corresponding to $\mathcal{E}_{\widehat{Z}}$ is the trivial representations.) Moreover, since $\det \mathcal{E}_{\widehat{Z}} \cong \mathcal{O}_{\widehat{Z}}(-K_{\widehat{Z}})$, we obtain
$$
\mathcal{O}_{\widehat{Z}}(-K_{\widehat{Z}}) \cong L^{\otimes (n+1)}.
$$
Since $-K_{\widehat{Z}}$ is ample, $L$ is also ample.
Therefore, by \cite[Corollary of Theorem 1.1]{KO73}, we conclude that $\widehat{Z} \cong \C\mathbb{P}^{n}$.
\end{proof}

\section{Example of a semistable vector bundle for a nef line bundle on a complex surface}
In this section, we give an example of semistable vector bundle on a complex surface which is semistable with respect to a nef and big line bundle, but it is not semistable with respect to any ample line bundles. 
Let $E$ be a holomorphic vector bundle on $\mathbb{C}P^2$ which is given by
\begin{equation}\label{def of E}
0\to\mathcal{O}_{\mathbb{C}P^2}\to E\to\mathcal{J}_{\{x_1,x_2\}}\otimes\mathcal{O}_{\mathbb{C}P^2}(2)\to0.
\end{equation}
By \cite[Chapter 2 \S 1.3]{OSS80}, above $E$ exists. Furthermore, $E$ is $c_1(\mathcal{O}(1))$-slope semistable, but not stable \cite[Chapter 2, \S 1.3, Example 1]{OSS80}.
Let $\pi:X\to \mathbb{C}P^2$ be a blow up at a point $p\in \mathbb{C}P^2$ and $C$ be the $\pi$-exceptional divisor. 
\begin{lem}
    For any ample line bundle $H$ on $X$, there are positive integer $a>0$ and $b>0$ such that $H=\mathcal{O}(a)\otimes\mathcal{O}(-bC)$ and $a>b$.
\end{lem}
\begin{proof}
    Since $H^2(\mathbb{C}P^2,\mathbb{Z})=\mathbb{Z}\mathcal{O}(1)$, the ample cone of $\mathbb{C}P^2$ coincides with the big cone and it is described as  $\Amp(\mathbb{C}P^2)=\mathbb{Z}_{>0}\mathcal{O}(1)$. Let $H$ be an ample line bundle of $X$. Since $\pi_*H$ is big, there is a constant $a>0$ such that $\pi_*H=\mathcal{O}(a)$. Since ${\rm{Pic}}(X)=\pi^*{\rm{Pic}}(\mathbb{C}P^2)\oplus\mathbb{Z}[C]$, there is a constant $b'\in\mathbb{Z}$ such that $H=\pi^*\mathcal{O}(a)\otimes\mathcal{O}(b'C)$. If $b'>0$, then the intersection $(H\cdot C)=-b'<0$, which contradicts to ampleness of $H$. Thus $b'=-b<0$. Since $0<(H^2)=a^2-b^2$, we obtain that $a>b$.
\end{proof}
\begin{prop}\label{eg nef}
    Let us define a holomorphic vector bundle $F$ on $X$ by $F:=\pi^*E\oplus(\mathcal{O}(C)\otimes\pi^*\mathcal{O}(1))$. Then the following holds:\\
    $(1)$ $F$ is $\pi^*\mathcal{O}(1)$-semistable, but not stable.\\
    $(2)$ For any ample line bundle $H$ on $X$, $F$ is not $H$-semistable.\\
    $(3)$ $F$ satisfies the Bogomolov-Gieseker inequality.
\end{prop}
\begin{proof}
    $(1)$ We can easily see that $\mu_{\mathcal{O}(1)}(E)=1=\mu_{\mathcal{O}(1)}(\mathcal{O}(1))$ and recall that both $E$ and $\mathcal{O}(1)$ are $\mathcal{O}(1)$-semistable. Thus $\pi_{[*]}E=E\oplus \mathcal{O}(1)$ is $\mathcal{O}(1)$-semistable (c.f. \cite{Kob87} Chapter 5, Proposition 7.9). Since semistability is bimeromorphic invariant by \cite[Proposition 3.28]{IJZ25}, $F$ is $\pi^*\mathcal{O}(1)$-semistable. And $F$ is not $\pi^*\mathcal{O}(1)$-stable since a subbundle $\pi^*E\subset F$ satisfies $\mu_{\pi^*\mathcal{O}(1)}(\pi^*E)=1=\mu_{\pi^*\mathcal{O}(1)}(F)$.\\
    $(2)$ Let $H$ be an ample line bundle on $X$. We normalize $H$ so that $H=\mathcal{O}(1)\otimes\mathcal{O}(-bC)$ with $b\in \mathbb{Q}_{>0}$. We know $0<b<1$. We will show that $F$ is not $H$-semistable. The $H$-slope of $F$ is computed as follows:
    \begin{align*}
    \mu_H(F)
    &=\frac{1}{3}\left(\deg_H(\pi^*E)+\deg_H(C)+\deg_H(\pi^*\mathcal{O}(1))\right)\\
    &=\frac{1}{3}(c_1(E)\cdot c_1(\mathcal{O}(1))+c_1(C)\cdot(-bC)+c_1(\mathcal{O}(1))^2)\\
    &=1+\frac{b}{3}.
    \end{align*}
    On the other hand, the $H$-slope of a rank 1 subbundle $\mathcal{O}(C)\otimes\pi^*\mathcal{O}(1)$ is as follows:
    $$
    \mu_H(\mathcal{O}(C)\otimes\pi^*\mathcal{O}(1))=c_1(C)\cdot (-bC)+c_1(\pi^*\mathcal{O}(1))^2=b+1.
    $$
    Therefore $\mathcal{O}(C)\otimes\pi^*\mathcal{O}(1)$ is a $H$-destabilizing subbundle of $F$.\\
    $(3)$ It is a consequence of Theorem \ref{BG ineq}. We can show by the following direct computation.
    By the definition of $F$, we have
    \begin{itemize}
    \item $c_1(F)=c_1(\pi^*E)+c_1(\pi^*\mathcal{O}(1))+c_1(C)$
    \item $c_2(F)=c_2(\pi^*E)+c_1(\pi^*E)\cdot (c_1(\pi^*\mathcal{O}(1)+c_1(C))$
    \end{itemize}
    By (\ref{def of E}), we have $c_1(\pi^*E)=c_1(\pi^*\mathcal{O}(1))$ and $c_2(E)=c_2(\mathcal{J}_{\{x_1,x_2\}})=2$ (cf. \cite[Chapter 1, \S 5, Example 1]{OSS80}). Thus we have $c_1(F)^2=8$ and $c_2(F)=4$. Then 
    $$
    2rc_2(F)-(r-1)c_1(F)^2=2\cdot3\cdot4-(3-1)\cdot8=24-16=8>0.
    $$
\end{proof}


\begin{thebibliography}{99}
\bibitem[Bou04]{Bou04}
S. Boucksom, Divisorial Zariski decompositions on compact complex manifolds, Ann. Sci. \'Ecole Norm. Sup. (4) {\bf 37} (2004), no.~1, 45--76; MR2050205.
\bibitem[BDPP13]{BDPP13}
S. Boucksom, J.-P. Demailly, M. Paun and T. Peternell, The pseudo-effective cone of a compact K\"ahler manifold and varieties of negative Kodaira dimension, J. Algebraic Geom. {\bf 22} (2013), no.~2, 201--248; MR3019449.
\bibitem[BEGZ10]{BEGZ10}
S. Boucksom, P. Eyssidieux, V. Guedji, and A. Zeriahi, Monge-Amp\`ere equations in big cohomology classes, Acta Math. {\bf 205} (2010), no.~2, 199--262; MR2746347.
\bibitem[BG14]{BG14}
R.~J. Berman and H. Guenancia, K\"ahler-Einstein metrics on stable varieties and log canonical pairs, Geom. Funct. Anal. {\bf 24} (2014), no.~6, 1683--1730; MR3283927.
\bibitem[CGNPPW23]{CGNPPW23}
J. Cao, P. Graf, P. Naumann, M. Paun, T. Peternell, and X. Wu, Hermite--Einstein metrics in singular settings. arXiv preprint arXiv:2303.08773 (2023).
\bibitem[Chen25]{Chen25}
X. Chen, Admissible Hermitian–Yang–Mills connections over normal varieties. Mathematische Annalen (2025): 1-37.
\bibitem[CCHTST25]{CCHTST25}
Y. Chen, Y. Chiu, K., M. Hallgren, G., Tô, T. D. Székelyhidi, and F. Tong, (2025). On K\" ahler-Einstein Currents. arXiv preprint arXiv:2502.09825.
\bibitem[CW22]{CW22}
X. Chen and R.~A. Wentworth, Compactness for $\Omega$-Yang-Mills connections, Calc. Var. Partial Differential Equations {\bf 61} (2022), no.~2, Paper No. 58, 30 pp.; MR4376537.
\bibitem[DGP24]{DGP24}
S. Druel, H. Guenancia and M. P\u aun, A decomposition theorem for $\Bbb Q$-Fano K\"ahler-Einstein varieties, C. R. Math. Acad. Sci. Paris {\bf 362} (2024), Special issue no. S1, 93--118; MR4762190.
\bibitem[DHY23]{DHY23}
O. Das, C. Hacon, and J. I. Yáñez, MMP for Generalized Pairs on K\" ahler 3-folds. arXiv preprint arXiv:2305.00524 (2023).
\bibitem[Faulk22]{Faulk22}
M. Faulk, Hermitian-Einstein metrics on stable vector bundles over compact K\" ahler orbifolds. (2022). arXiv preprint arXiv:2202.08885.
\bibitem[FO25]{FO25}
X. Fu, W. Ou, Orbifold Bogomolov-Gieseker inequalities on compact K\" ahler varieties. (2025) arXiv preprint arXiv:2511.03530.
\bibitem[GKKP11]{GKKP11}
D. Greb et al., Differential forms on log canonical spaces, Publ. Math. Inst. Hautes \'Etudes Sci. No. 114 (2011), 87--169; MR2854859.
\bibitem[GKP16]{GKP16}
D. Greb, S. Kebekus and T. Peternell, \'Etale fundamental groups of Kawamata log terminal spaces, flat sheaves, and quotients of abelian varieties, Duke Math. J. {\bf 165} (2016), no.~10, 1965--2004; MR3522654.
\bibitem[GKP22]{GKP22}
D. Greb, S. Kebekus and T. Peternell, Projective flatness over klt spaces and uniformisation of varieties with nef anti-canonical divisor, J. Algebraic Geom. {\bf 31} (2022), no.~3, 467--496; MR4484547.
\bibitem[GKPT19]{GKPT19}
D. Greb, S. Kebekus, T. Peternell, B. Taji, The Miyaoka-Yau inequality and uniformisation of canonical models, Ann. Sci. \'Ec. Norm. Sup\'er. (4) {\bf 52} (2019), no.~6, 1487--1535; MR4061021.
\bibitem[GPP25]{GPP25}
H. Guenancia, C. M. Pan, M. Păun,. A note on orbifold regularity of canonical metrics. (2025) arXiv preprint arXiv:2509.09259.
\bibitem[GPSS23]{GPSS23}
B. Guo, D.H. Phong, J. Song, and J. Sturm, Sobolev inequalities on K\" ahler spaces. (2023) arXiv preprint arXiv:2311.00221.
\bibitem[GT22]{GT22}
H. Guenancia and B. Taji, Orbifold stability and Miyaoka-Yau inequality for minimal pairs, Geom. Topol. {\bf 26} (2022), no.~4, 1435--1482; MR4504444.
\bibitem[Hiro75]{Hiro}
H. Hironaka, Flattening theorem in complex-analytic geometry, Amer. J. Math. {\bf 97} (1975), 503--547; MR0393556.
\bibitem[His24]{His24}
T. Hisamoto, On the Miyaoka-Yau inequality for manifolds with nef anti-canonical line bundle. arXiv preprint (2024) arXiv:2403.09120.
\bibitem[IJZ25]{IJZ25}
M. Iwai, S. Jinnouchi, S. Zhang, The Miyaoka-Yau inequality for singular varieties with big canonical or anticanonical divisors. (2025) arXiv preprint arXiv:2507.08522.
\bibitem[Jin25]{Jin25}
S. Jinnouchi, "Slope Stable Sheaves and Hermitian-Einstein Metrics on Normal Varieties with Big Cohomology Classes." (2025) arXiv preprint arXiv:2501.04910.
\bibitem[Jin25-2]{Jin25-2}
S. Jinnouchi, On the Kobayashi-Hitchin correspondence for K\"{a} hler currents. (2025) arXiv preprint arXiv:2511.20033.
\bibitem[JL25]{JL25} 
T. Jiang,  J. Li, Kobayashi-Hitchin Correspondence for Saturated Reflexive Parabolic Sheaves on K\" ahler manifolds. (2025) arXiv preprint arXiv:2506.03579.
\bibitem[Kob87]{Kob87}
S. Kobayashi, {\it Differential geometry of complex vector bundles}, Publications of the Mathematical Society of Japan Kan\^o{} Memorial Lectures, 15 5, Princeton Univ. Press, Princeton, NJ, 1987 Princeton Univ. Press, Princeton, NJ, 1987; MR0909698.
\bibitem[KO73]{KO73}
S. Kobayashi and T. Ochiai, Characterizations of complex projective spaces and hyperquadrics, J. Math. Kyoto Univ. {\bf 13} (1973), 31--47; MR0316745.
\bibitem[Li00]{Li00}
J.~Y. Li, Hermitian-Einstein metrics and Chern number inequalities on parabolic stable bundles over K\"ahler manifolds, Comm. Anal. Geom. {\bf 8} (2000), no.~3, 445--475; MR1775134.
\bibitem[Miy77]{Miy77}
Y. Miyaoka, On the Chern numbers of surfaces of general type, Invent. Math. {\bf 42} (1977), 225--237; MR0460343.
I
\bibitem[Nystr19]{Nystr19}
D. Witt~Nystr\"om, Duality between the pseudoeffective and the movable cone on a projective manifold, J. Amer. Math. Soc. {\bf 32} (2019), no.~3, 675--689; MR3981985.
\bibitem[OSS80]{OSS80}
C. Okonek, M.~H. Schneider and H. Spindler, {\it Vector bundles on complex projective spaces}, Progress in Mathematics, 3, Birkh\"auser, Boston, MA, 1980; MR0561910.
\bibitem[Ou25]{Ou25}
W. Ou, A characterization of uniruled compact K\" ahler manifolds. (2025) arXiv preprint arXiv:2501.18088.
\bibitem[UY86]{UY86}
K.~K. Uhlenbeck and S.-T. Yau, On the existence of Hermitian-Yang-Mills connections in stable vector bundles, Comm. Pure Appl. Math. {\bf 39} (1986), no. S, {\rm S}257--{\rm S}293; MR0861491.
\bibitem[Vu23]{Vu23}
Vu, D. V. Derivative of volumes of big cohomology classes. (2023), arXiv preprint arXiv:2307.15909.
\bibitem[Xu23]{Xu23}
C. Xu, K-stability for varieties with a big anticanonical class, \'Epijournal G\'eom. Alg\'ebrique {\bf [2023--2025]}, Special volume in honour of Claire Voisin, Art. 7, 9 pp.; MR4671736.
\bibitem[Yau77]{Yau77}
S.-T. Yau, Calabi's conjecture and some new results in algebraic geometry, Proc. Nat. Acad. Sci. U.S.A. {\bf 74} (1977), no.~5, 1798--1799; MR0451180.
\end{thebibliography}
\end{document}